\documentclass[psamsfonts,10pt, draft]{amsart}
\usepackage{amssymb,amsmath}
\usepackage{amsthm}
\usepackage{mathrsfs}
\usepackage{enumerate,indentfirst}
\usepackage[fleqn,tbtags]{mathtools}
\usepackage{bbm}
\usepackage{color}
\usepackage{courier}
\usepackage{hyperref}
\usepackage{cite}
\usepackage{tikz}

\usepackage{multicol}
\usepackage{float}

\newtheorem*{theorem*}{Main Theorem}
\newtheorem*{theorem**}{Theorem}

 \newtheorem{theorem}{Theorem}[section]
 \newtheorem{corollary}[theorem]{Corollary}
 \newtheorem{lemma}[theorem]{Lemma}
 \newtheorem{proposition}[theorem]{Proposition}
 
 \theoremstyle{definition}
 \newtheorem{example}[theorem]{Example}
 
 \newtheorem{remark}[theorem]{Remark}
 \theoremstyle{remark}
  \numberwithin{equation}{section}

\newcommand{\wh}{\widetilde{w}}
\newcommand{\tetah}{\widetilde{\vartheta}}

\newcommand{\xk}{\hat{x}}
\newcommand{\xv}{\check{x}}
\newcommand{\ck}{\hat{c}}
\newcommand{\lk}{\hat{\lambda}}
\newcommand{\bx}{\mathscr{B}_X}
\newcommand{\me}{\mathcal{P}_X}
\newcommand{\Fx}{\mathcal{F}_X}

\newcommand{\spm}{S_{\mu}}
\newcommand{\spn}{S_{\nu}}
\newcommand{\spvt}{S_{\vartheta}}

\newcommand{\bex}{\overline{B_{\varepsilon}(x)}}
\newcommand{\brx}{\overline{B_{\rho_r}(x)}}
\newcommand{\bkx}{\overline{B_{\rho_k}(x)}}

\newcommand{\tteta}{\widetilde{\vartheta}}
\newcommand{\ww}{\widetilde{w}}
\allowdisplaybreaks
\begin{document}

\title[A char. of isometries with respect to the L{\'e}vy--Prokhorov metric]{A characterisation of isometries with respect to the L{\'e}vy--Prokhorov metric}

\author[Gy.P. Geh{\'e}r]{Gy{\"o}rgy P{\'a}l Geh{\'e}r}

\address{%
Gy{\"o}rgy P{\'a}l Geh{\'e}r\newline 
MTA-SZTE Analysis and Stochastics Research Group \newline 
Bolyai Institute\\ University of Szeged\newline 
Aradi v{\'e}rtan{\'u}k tere 1.\\ Szeged H-6720\\ Hungary
\newline
\newline
MTA-DE \lq\lq Lend{\"u}let'' Functional Analysis Research Group \newline 
Institute of Mathematics\\ University of Debrecen \newline 
P.O. Box 12, Debrecen H-4010, Hungary}
\email{gehergyuri@gmail.com or gehergy@math.u-szeged.hu}

\author[T. Titkos]{Tam\'as Titkos}
\address{Tam{\'a}s Titkos\newline Alfr\'ed R\'enyi Institute of Mathematics, Hungarian Academy of Sciences \newline Re\'altanoda utca 13-15.\\ Budapest H-1053\\ Hungary}

\email{titkos.tamas@renyi.mta.hu}
\thanks{Gy{\"o}rgy P{\'a}l Geh{\'e}r was also supported by the \lq\lq Lend\" ulet'' Program (LP2012-46/2012) of the Hungarian Academy of Sciences, and by the Hungarian National Research, Development and Innovation Office -- NKFIH (grant no.~K115383).}
\thanks{Tam\'as Titkos was also supported by the \lq\lq Lend\" ulet'' Program (LP2012-46/2012) of the Hungarian Academy of Sciences, and by the Hungarian National Research, Development and Innovation Office -- NKFIH (grant no. ~K104206).}


\maketitle

\begin{abstract} 

According to the fundamental work of Yu.V. Prokhorov, the general theory of stochastic processes can be regarded as the theory of probability measures in complete separable metric spaces.
Since stochastic processes depending upon a continuous parameter are basically probability measures on certain subspaces of the space of all functions of a 
real variable, a particularly important case of this theory is when the underlying metric space has a linear structure. 
Prokhorov also provided a concrete metrisation of the topology of weak convergence today known as the L{\'e}vy--Prokhorov distance. 
Motivated by these facts, the famous Banach--Stone theorem, and some recent works related to characterisations of onto isometries of spaces of Borel probability measures, here we give a complete description of surjective isometries with respect to the L{\'e}vy--Prokhorov metric in case when the underlying metric space is a separable Banach space.
Our result can be considered as a generalisation of L. Moln\'ar's earlier Banach--Stone-type result which characterises onto isometries of the space of all probability distribution functions on the real line wit respect to the L{\'e}vy distance.
However, the present more general setting requires the development of an essentially new technique.

\end{abstract}


\section{Introduction}
There is a long history and vast literature of isometries (i.e. not necessarily sujective distance preserving maps) on different kind of metric spaces.
Two classical results in the case of normed linear spaces are the Mazur--Ulam theorem which states that every surjective isometry between real normed spaces is automatically affine (i.e. linear up to translation), and the Banach--Stone theorem which provides the structure of onto linear isometries between Banach spaces of continuous scalar-valued functions on compact Hausdorff spaces.
Since then several properties of surjective linear isometries on different types of normed spaces have been explored, see for instance the papers \cite{lin1,lin2,kelle7,kelle1,lin4,kelle3,lin8} and the extensive books \cite{FJ1,FJ2}.
The reader can find similar results on non-linear spaces for example in \cite{kelle2,BJM,GS,kelle4,kelle5,kelle6}.

The starting point of our investigation is Moln\'ar's paper \cite{ML-Levy} where a complete description of surjective L{\'e}vy isometries of the non-linear space $\mathcal{D}(\mathbb{R})$ of all cumulative distribution functions was given. 
If $F,G\in\mathcal{D}(\mathbb{R})$, then their \emph{L{\'e}vy distance} is defined by the following formula:
\begin{align*}
L(F,G) := \inf\left\{\varepsilon>0\,\big|\,\forall\; t\in\mathbb{R}\colon F(t-\varepsilon)-\varepsilon\leq G(t)\leq F(t+\varepsilon)+\varepsilon\right\}.
\end{align*}
The importance of this metric lies in the fact that it metrises the topology of weak convergence on $\mathcal{D}(\mathbb{R})$.
Moln\'ar's result reads as follows (see \cite[Theorem 1]{ML-Levy}): 
\emph{let $\Phi\colon \mathcal{D}(\mathbb{R})\to \mathcal{D}(\mathbb{R})$ be a surjective L{\'e}vy isometry, i.e., a bijective map satisfying
\begin{align*}
L(F,G)=L(\Phi(F),\Phi(G)) \qquad (\forall\;F,G\in \mathcal{D}(\mathbb{R})).
\end{align*}
Then there is a constant $c\in\mathbb{R}$ such that $\Phi$ is one of the following two forms:
\begin{align*}
\Phi(F)(t)=F(t+c)\qquad(\forall\; t\in\mathbb{R}, F\in\mathcal{D}(\mathbb{R})),
\end{align*}
or
\begin{align*}
\Phi(F)(t)=1-\lim\limits_{s\to t-}F(-s+c)\qquad(\forall\; t\in\mathbb{R}, F\in\mathcal{D}(\mathbb{R})).
\end{align*}}
In other words, every surjective L{\'e}vy isometry is induced by an isometry of $\mathbb{R}$ with respect to its usual norm (or equivalently, by a composition of a translation and a reflection on $\mathbb{R}$).

The investigation of surjective isometries on spaces of Borel probability measures was continued for example in \cite{DoMo,ML2} for the Kolmogorov--Smirnov distance which is important in the Kolmogorov--Smirnov statistic and test, and in \cite{Wasserstein2,Wasserstein3,Wasserstein} with respect to the Wasserstein (or Kantorovich) metric which metrises the weak convergence.

Let $(X,d)$ be a complete and separable metric space.
We will denote the \emph{$\sigma$-algebra of Borel sets} on $X$ by $\mathscr{B}_X$ and the \emph{set of all Borel probability measures} by $\me$.
The L{\'e}vy distance gives a metrisation of weak convergence on $\mathcal{D}({\mathbb{R}})$, or equivalently on $\mathcal{P}_{\mathbb{R}}$.
In 1956 Prokhorov managed to metrise the weak convergence of $\me$ for general complete and separable metric spaces $(X,d)$.
The so-called \emph{L{\'e}vy--Prokhorov distance} which was introduced by him in \cite{Pr} is defined by
\begin{equation}\label{LPdef}
\pi(\mu,\nu):= \inf\left\{\varepsilon>0\,\big|\,\forall\; A\in\bx\colon\mu(A)\leq\nu(A^{\varepsilon})+\varepsilon\right\}, 
\end{equation}
where 
\begin{equation*}
A^{\varepsilon} := \bigcup\limits_{x\in A} B_\varepsilon(x) \qquad \text{and} \qquad B_\varepsilon(x) := \{z\in X\,|\,d(x,z)<\varepsilon\}.
\end{equation*}
For the details and elementary properties see e.g. \cite[p. 27]{Hu}.
Let us point out that in the special case when $X = \mathbb{R}$ this metric differs from the original L{\'e}vy distance.
Here arises the following very natural question: 
\begin{center}
\emph{What is the structure of onto isometries with respect to the L{\'e}vy--Prokhorov metric on $\me$ if $X$ is a general separable real Banach space?}
\end{center}
This paper is devoted to give an answer to this question. 
Namely, we will prove that every such transformation is induced by an affine isometry of the underlying space $X$.

There are some particularly important cases in our investigation which we emphasise now.
Namely, since stochastic processes depending upon a continuous parameter are basically probability measures on certain subspaces of the space of all functions of a real variable (see e.g. \cite{Ambrose,Doob}), one particularly interesting case is when the underlying Banach space is $C([0,1])$, i.e. the space of all continuous real-valued functions on $[0,1]$ endowed with the uniform norm $\|\cdot\|_{\infty}$. 
For details see \cite[Chapter 2]{Pr} or \cite[Chapter 2]{Bi}.
Further two important cases are when $X$ is a Euclidean space because of multivariate random variables, and when $X$ is an infinite dimensional, separable real Hilbert space because of the theory of random elements in Hilbert spaces.


\section{The setting and the statment of our main result}\label{2}
In this section we state the main result of the paper and collect some definitions and well-known facts about weak convergence of Borel probability measures.
For more details the reader is referred to the textbooks of Billingsley \cite{Bi}, Huber \cite{Hu} and Parthasarathy \cite{Pa}.

Let $(X,d)$ be a complete, separable metric space and denote by $C_b(X;\mathbb{R})$ the Banach space of all real-valued bounded continuous functions.
Recall that $\bx$ is the smallest $\sigma$-algebra with respect to each $f\in C_b(X;\mathbb{R})$ is measurable.
We say that an element of $\me$ is a \emph{Dirac measure} if it is concentrated on one point, and for an $x\in X$ the symbol $\delta_x$ stands for the corresponding Dirac measure.
The set of all Dirac measures on $X$ is denoted by $\Delta_X$. 
The collection of all \emph{finitely supported measures} is
\begin{equation*}
\mathcal{F}_X := 
\left\{\sum\limits_{i\in I}\lambda_i\delta_{x_i}\,\Big|\,\# I<\aleph_0, ~\sum\limits_{i\in I}\lambda_i=1,~\lambda_i>0,\,x_i\in X~(\forall\;i\in I)\right\},
\end{equation*}
which is actually the convex hull of $\Delta_X$. 
The \emph{support} (or spectrum) of $\mu\in\me$ is the smallest $d$-closed set $S_{\mu}$ that satisfies $\mu(S_{\mu}) = 1$. 
Moreover, it is not hard to verify the following equation:
\begin{equation*}
S_{\mu} = \left\{x\in X\,\big|\,\forall\; r>0\colon \mu(B_r(x))>0\right\}.
\end{equation*}
The \emph{closure} of a set $H\subseteq X$ will be denoted by $\overline{H}$. 

We say that a sequence of measures $\{\mu_n\}_{n=1}^\infty \subset \me$ \emph{converges weakly} to a $\mu\in\me$ if we have
$$
\int f~d\mu_n \to \int f~d\mu\qquad (\forall\; f\in C_b(X;\mathbb{R})).
$$
This type of convergence is metrised by the L{\'e}vy--Prokhorov metric given by \eqref{LPdef}.
A map $\varphi\colon\me\to\me$ is called a \emph{$\pi$-isometry} on $\me$ if
$$
\pi(\mu,\nu) = \pi(\varphi(\mu),\varphi(\nu)) \qquad (\forall\;\mu,\nu\in\me)
$$
is satisfied.

Now, we are in the position to state the main result of this paper. 

\begin{theorem*}
Let $(X,\|\cdot\|)$ be a separable real Banach space and $\varphi\colon\me\to\me$ be a surjective $\pi$-isometry.
Then there exists a surjective affine isometry $\psi\colon X\to X$ which induces $\varphi$, i.e. we have
\begin{align}\label{MT IMPL}
\left(\varphi(\mu)\right)(A) = \mu(\psi^{-1}[A]) \qquad (\forall\; A\in\bx),
\end{align}
where $\psi^{-1}[A]$ denotes the inverse-image set $\{\psi^{-1}(a)\,|\,a\in A\}$.
\end{theorem*}

The converse of the above statement is trivial, namely, every transformation of the form \eqref{MT IMPL} is obviously an onto $\pi$-isometry.
Note that our theorem can be re-phrased in terms of push-forward measures.
Namely, the action of $\varphi$ is just the push-forward with respect to the isometry $\psi\colon X\to X$.


As we already mentioned, the L{\'e}vy--Prokhorov metric on $\mathcal{P}_\mathbb{R}$ differs from the L{\'e}vy distance on $\mathcal{P}_\mathbb{R}$. 
Therefore, in the special case when $X = \mathbb{R}$, our hypotheses are different from those given in \cite[Theorem 1]{ML-Levy}, although the conclusion is the same.

Our proof is given in the next section where we will have four major steps.
This will be followed by some remarks in the final section, where we will also point out that our Main Theorem still holds if we replace $\me$ with an arbitrary weakly dense subset $\mathcal{S}$.


\section{Proof}\label{3}

The proof is divided into four major steps.
First, we will explore the action of $\varphi$ on $\Delta_X$. 
Then for finitely supported measures $\mu$ we will investigate the behaviour of its image $\varphi(\mu)$ near to the vertices of the convex hull of $S_\mu$.
This will be followed by providing a procedure which will allow us to obtain important information about the \lq\lq rest'' of $\varphi(\mu)$.
Finally, we close this section with the proof of the Main Theorem.
Note that although our main result deals with Borel probability measures on separable Banach spaces, we state and prove some results in the context of complete and separable metric spaces. 


\subsection{First major step: the action on Dirac measures}

Here we will investigate properties of the restricted map $\varphi |_{\mathcal{D}_X}$.
Namely, we will prove that $\varphi$ maps $\Delta_X$ onto $\Delta_X$, furthermore, there is a surjective affine isometry of $X$ which induces this restriction.
In order to do this first, we formulate the metric phrase \lq\lq distance one'' by means of the supports of measures.

\begin{proposition}\label{1-tav}
Let $(X,d)$ be a complete, separable metric space and $\mu,\nu\in\me$. Then the following statements are equivalent:
\begin{itemize}
\item[(i)] $\pi(\mu,\nu)=1$,
\item[(ii)] $\underline{\mathrm{d}}(S_{\mu},S_{\nu}) := \inf\left\{d(x,y)\,|\,x\in S_{\mu},~y\in S_{\nu}\right\}\geq 1$,
\item[(iii)] $S_{\nu}\cap S_{\mu}^1=\emptyset$,
\item[(iv)] $S_{\mu}\cap S_{\nu}^1=\emptyset$.
\end{itemize}
\end{proposition}

\begin{proof}
Observe that (ii) implies the following inequality for every $0<\varepsilon<1$:
\begin{equation*}
1=\mu(\spm)>\nu(\spm^{\varepsilon})+\varepsilon=\varepsilon.
\end{equation*}
Consequently we have $\pi(\mu,\nu)\geq 1$. But on the other hand, $\pi(\mu,\nu)\leq1$ holds for all $\mu,\nu\in\me$, and therefore the (ii)$\Rightarrow$(i) part is complete. 

To prove (i)$\Rightarrow$(ii) assume that $\varrho := \underline{\mathrm{d}}(\spm,\spn) < 1$. In this case one can fix two points $x^*\in\spm$ and $y^*\in\spn$, and a positive number $r>0$ which satisfy both 
\begin{align*}
\varrho\leq d(x^*,y^*)=:\varrho'<1\quad\mbox{and}\quad\varrho'+2r<1.
\end{align*}
We also set $t := \min\left\{\mu\left(B_r(x^*)\right),\nu\left(B_r(y^*)\right)\right\}$ which is clearly positive by the very definition of the support. We will show that $\hat\varepsilon:=\max\{1-t,\varrho'+2r\}<1$ is a suitable choice to guarantee 
\begin{equation*}
\mu(A)\leq\nu(A^{\hat\varepsilon})+\hat\varepsilon \qquad (\forall\; A\in\bx)
\end{equation*}
Indeed, if $A\in\bx$ satisfies $\mu(A)\leq 1-t$, then
\begin{equation*}
\mu(A)\leq 1-t\leq\nu(A^{1-t})+1-t\leq\nu(A^{\hat\varepsilon})+\hat\varepsilon.
\end{equation*}
On the other hand, if $\mu(A)>1-t$, then we observe that $\mu(A\cap B_r(x^*))>0$, and consequently $A\cap B_r(x^*)$ is not empty. 
Let us fix a point $z\in A\cap B_r(x^*)$. 
Using the triangle inequality we infer $d(y^*,z)\leq d(y^*,x^*)+d(x^*,z) < \varrho'+r$ and $B_r(y^*)\subseteq B_{\varrho'+2r}(z)\subseteq A^{\hat{\varepsilon}}$. 
Therefore we conclude
$$
\mu(A)\leq 1\leq t+\hat{\varepsilon}\leq\nu(B_r(y^*))+\hat{\varepsilon}\leq\nu(A^{\hat{\varepsilon}})+\hat{\varepsilon},
$$
which implies $\pi(\mu,\nu)\leq\hat{\varepsilon}<1$.

The equivalence of (ii), (iii) and (iv) follows from the definitions.
\end{proof}

Next, let us define the \emph{unit distance set} of a set of measures $\mathscr{A} \subseteq\me$ by
\begin{equation*}
\mathscr{A}^{\mathbbm{u}} = \left\{\nu\in\me\,\big|\,\forall\;\mu\in\mathscr{A}\colon\pi(\mu,\nu)=1\right\}.
\end{equation*}
(Remark that by definition we have $\emptyset^{\mathbbm{u}} = \me$.)
The following statement gives a metric characterisation of Dirac measures when $X$ is a separable real Banach space.
We point out that similar results were also crucial ideas in \cite{DoMo,ML-Levy,ML2}.

\begin{proposition}\label{L:metric char}
Let $(X,d)$ be a complete, separable metric space and $\mu\in\me$ be an arbitrary Borel probability measure on it. 
Then the following three statements are equivalent:
\begin{itemize}
\item[(i)] $\left(\{\mu\}^{\mathbbm{u}}\right)^{\mathbbm{u}}=\{\mu\}$,
\item[(ii)] there exists an $x\in X$ such that
\begin{itemize}
\item[(ii/a)] $\mu=\delta_x$, and
\item[(ii/b)] $B_1(y)\subseteq B_1(x)$ implies $x=y$ for every $y\in X$,
\end{itemize}
\item[(iii)] $\#\left(\{\mu\}^{\mathbbm{u}}\right)^{\mathbbm{u}} = 1$.
\end{itemize}
\end{proposition}

\begin{proof}
First, let us characterise the elements of $\left(\{\mu\}^{\mathbbm{u}}\right)^{\mathbbm{u}}$. 
It follows from Proposition \ref{1-tav} that
\begin{equation*}
\{\mu\}^{\mathbbm{u}} = \left\{ \nu\in \me \,\big|\,\spn^1 \cap S_{\mu} = \emptyset\right\} = \left\{ \nu\in \me \,\big|\,\spn \cap S_{\mu}^{1} = \emptyset\right\}.
\end{equation*}
Applying this observation twice, we easily see that
\begin{equation}\label{emptycap}
\left(\{\mu\}^{\mathbbm{u}}\right)^{\mathbbm{u}} = \bigcap_{\nu\in \{\mu\}^{\mathbbm{u}}} \{\nu\}^{\mathbbm{u}} = \bigcap_{\nu\in\me, S_{\nu}\cap S_\mu^1 = \emptyset} \left\{ \vartheta\in\me \,\big|\, S_{\nu}\cap S_{\vartheta}^1 = \emptyset \right\},
\end{equation}
and therefore we obtain the following equivalence:
\begin{equation}\label{suppjell}
\vartheta\in\left(\{\mu\}^{\mathbbm{u}}\right)^{\mathbbm{u}}
\qquad\Longleftrightarrow\qquad
\spvt^1 \subseteq S_{\mu}^1.
\end{equation}
(Note that if $\{\mu\}^{\mathbbm{u}} = \emptyset$, then $\bigcap_{\nu\in \emptyset} \{\nu\}^{\mathbbm{u}} = \me$ in \eqref{emptycap} by definition.)

Now, since $\mu\in \left(\{\mu\}^{\mathbbm{u}}\right)^{\mathbbm{u}}$ always holds, the equivalence of (i) and (iii) is apparent.
We continue with proving (i)$\Rightarrow$(ii). 
Observe that \eqref{suppjell} implies $$\left\{\delta_z \,\big|\, z\in S_\mu \right\} \subseteq \left(\{\mu\}^{\mathbbm{u}}\right)^{\mathbbm{u}},$$ thus (ii/a) follows.
On the other hand, if any $y\in X$ satisfies $$S_{\delta_y}^1 = B_1(y) \subseteq B_1(x) = S_{\delta_x}^1,$$ then again by \eqref{suppjell} we infer $x = y$.

Finally, we show (ii)$\Rightarrow$(i). Assume that $\vartheta\in\left(\{\mu\}^{\mathbbm{u}}\right)^{\mathbbm{u}}$.
By (ii/b) we get that $\spvt^1 \subseteq S_{\delta_x}^1$ holds if and only if $\spvt \subseteq \{x\}$, which implies (i).
\end{proof}

\begin{remark}
Note that if the diameter of the metric space $X$ is less than 1, i.e. there exists an $0<r<1$ such that $d(x,y) \leq r$ $(\forall\;x,y\in X)$, then $\pi(\mu,\nu) \leq r$ holds for every $\mu,\nu\in\me$.
In particular, $\{\mu\}^{\mathbbm{u}} = \emptyset$ and thus $\left(\{\mu\}^{\mathbbm{u}}\right)^{\mathbbm{u}} = \me$ for every $\mu\in\me$.
\end{remark}

The following lemma describes the action of $\varphi$ on Dirac measures.

\begin{lemma}\label{implementedisometry}
Let $(X,\|\cdot\|)$ be a separable real Banach space, and let $\varphi\colon\me\to \me$ be a surjective $\pi$ isometry. Then there exists a surjective affine isometry $\psi\colon X\to X$ such that
\begin{equation}\label{psi}
\varphi(\delta_x) = \delta_{\psi(x)} \qquad (\forall\;x\in X).
\end{equation}
\end{lemma}

\begin{proof} 
Since $\varphi$ is a bijective isometry, we have 
\begin{equation*}
\varphi\big(\left(\{\mu\}^{\mathbbm{u}}\right)^{\mathbbm{u}}\big) = \big(\{\varphi(\mu)\}^{\mathbbm{u}}\big)^{\mathbbm{u}} \qquad (\forall\;\mu\in\me).
\end{equation*}
Thus an easy application of the previous proposition yields $\varphi(\Delta_X) = \Delta_X$.
This also means that there exists a bijective map $\psi\colon X\to X$ which induces the restriction $\varphi |_{\Delta_X}$, i.e.
\begin{equation}
\varphi(\delta_x):=\delta_{\psi(x)} \qquad (\forall\;x\in X).
\end{equation}

We will show that $\psi$ is an isometry. Observe that 
$$
\pi(\delta_{x_1},\delta_{x_2}) = \min\{1,\|x_1-x_2\|\} \qquad (\forall\; x_1,x_2 \in X).
$$ 
Therefore for all $\alpha\in(0,1)$ we have
\begin{equation}\label{lokizom}
\|\psi(x_1)-\psi(x_2)\|=\alpha \;\iff\; \|x_1-x_2\|=\alpha \qquad (\forall\; x_1,x_2\in X).
\end{equation}
If $X$ is one-dimensional, then it is rather easy to see that \eqref{lokizom} implies the isometriness of $\psi$. 
Now, assume that $\dim X \geq 2$. 
After suitable renorming (i.e. considering the norm $|||\cdot||| := \tfrac{1}{\alpha}\|\cdot\|$), from a result of T.M. Rassias and P. \v Semrl \cite[Theorem 1]{RaSe} we conclude that
\begin{align*}
\|\psi(x)-\psi(y)\|=n\alpha\;\iff\;\|x-y\|=n\alpha \qquad (\forall\;\alpha\in(0,1), n\in\mathbb{N}),
\end{align*}
and therefore $\psi$ is indeed an isometry. Finally, by the famous Mazur--Ulam theorem we obtain that $\psi$ is affine, which completes the proof.
\end{proof}

We remark that the last step of the proof (using the Rassias--\v Semrl theorem) can be also done by the extension theorem of Mankiewicz \cite{PM}.

In light of the above lemma, from now on we may and do assume without loss of generality that $\varphi$ acts identically on $\Delta_X$, i.e.,
\begin{align}\label{FIX}
\varphi(\delta_x)=\delta_x\qquad(\forall\; x\in X),
\end{align}
and our aim will be to show that $\varphi$ acts identically on the whole of $\me$.
After we do so, to obtain the result of our Main Theorem for general surjective $\pi$-isometries will be straightforward.
Namely, if \eqref{psi} is fulfilled, then we can consider the following modified transformation: 
\begin{equation}\label{eq:id}
\varphi_{\psi}\colon\me\to\me, \quad \left(\varphi_{\psi}(\mu)\right)(A) := \left(\varphi(\mu)\right)(\psi[A])\qquad(\forall\;\mu\in\me, A\in\mathscr{B}).
\end{equation}
By our assumption, $\varphi_{\psi}$ fixes every element of $\Delta_X$ and thus also of $\me$, which implies \eqref{MT IMPL}.

Next, let us define the following continuous function for each $\mu\in\me$: 
$$
W_{\mu}\colon X\to[0,1], \;\; W_{\mu}(x) := \pi(\delta_x,\mu)
$$
which will be called the \emph{witness function} of $\mu$.
The main advantage of the assumption \eqref{FIX} is that the witness function becomes $\varphi$-invariant, i.e.
\begin{equation}\label{invariancia}
W_{\mu}(x)=\pi(\delta_x,\mu)=\pi(\varphi(\delta_x),\varphi(\mu))=\pi(\delta_x,\varphi(\mu))=W_{\varphi(\mu)}(x)\quad (\forall\; x\in X).
\end{equation}
It is natural to expect that the shape of the witness function carries some information about the measure.
The last three major steps of the proof will be devoted to explore this for the $\varphi$-images of finitely supported measures in the setting of separable Banach spaces.
However, as demonstrated by the next example, the witness function usually does not distinguish measures in general complete and separable metric spaces.

\begin{example}
Consider the complete and separable metric space $(X,d)$ with
\begin{equation*}
X:=\{x_1,x_2,x_3\} \qquad\text{and}\qquad d(x_i,x_j) := 
\left\{
\begin{matrix}
\frac{1}{3} & \text{if } i \neq j\\
0 & \text{if } i=j
\end{matrix}
\right..
\end{equation*}
Let $\mu:=\frac{1}{2}\delta_{x_1}+\frac{1}{2}\delta_{x_2}$ and $\nu:=\frac{1}{2}\delta_{x_2}+\frac{1}{2}\delta_{x_3}$. 
An easy calculation shows that we have $\pi(\delta_x,\mu)=\frac{1}{3}=\pi(\delta_x,\nu)$ for all $x\in X$ and hence $W_{\mu}\equiv W_{\nu}$.
\end{example}


\subsection{Second major step: isolated atoms on the vertices of the convex hull of the support}

Here we will prove that if $\mu$ is a finitely supported measure, and $\xk$ is a vertex of the convex hull of $\spm$, then $\xk$ is an isolated atom of $\varphi(\mu)$ and
\begin{align*}
\mu(\{\xk\}) = (\varphi(\mu))(\{\xk\}).
\end{align*}

We begin with a technical statement, which will be very useful in the sequel.

\begin{proposition}\label{min}
	Let $X$ be a separable real Banach space, and suppose that $\mu$ is a finitely supported measure.
	Then for every $\nu\in\me$, $\nu\neq\mu$ we have
	\begin{align}
		\label{min formula}
		\pi(\mu,\nu) = \min\left\{\varepsilon > 0\,|\,\forall\; A\subseteq S_\mu\colon \mu(A)\leq \nu(\overline{A^{\varepsilon}})+\varepsilon\right\}.
	\end{align}
\end{proposition}

\begin{proof}
	First, we observe that in \eqref{LPdef} it is enough to consider Borel sets satisfying $A\subseteq S_{\mu}$. 
	Furthermore, it is obvious that
	$$
	\pi(\mu,\nu)=\max\left\{\inf\{\varepsilon > 0\,|\, \mu(A)\leq\nu(A^{\varepsilon})+\varepsilon\}\,\big|\,A\subseteq S_{\mu}\right\}.
	$$
	Therefore it is enough to show that for each subset $A := \{a_1,\dots,a_k\}\subseteq S_{\mu}$ if the infimum
	\begin{align*}
		\varepsilon_A & := \inf\left\{\varepsilon>0\,|\,\mu(A)\leq\nu(A^{\varepsilon})+\varepsilon\right\} \\
		& = \inf\left\{\varepsilon>0\,\Bigg|\,\mu(\{a_1,\dots,a_k\}) \leq \nu\Bigg(\bigcup_{j=1}^k B_{\varepsilon}(a_j)\Bigg)+\varepsilon\right\}
	\end{align*}
	is positive, then it is actually a minimum if we take the closure of $A^{\varepsilon}$ instead of $A^{\varepsilon}$.
	The following line of inequalities holds for all $0<h<\varepsilon_A$ by definition:
	$$
	\nu\Bigg(\bigcup_{j=1}^k B_{\varepsilon_{A}-h}(a_j)\Bigg)+\varepsilon_{A}-h<\mu\left(\{a_1,\dots,a_k\}\right) \leq \nu\Bigg(\bigcup_{j=1}^k B_{\varepsilon_{A}+h}(a_j)\Bigg)+\varepsilon_{A}+h.
	$$
	Now, taking the limit of the right-hand side as $h\to 0+$, and using that $0 < r < s$ implies $\overline{B_r(x)}\subseteq B_s(x)$, we obtain
	$$
	\nu\Bigg(\bigcup_{j=1}^k \overline{B_{\varepsilon_{A}-\delta}(a_j)}\Bigg)+\varepsilon_{A}-\delta<\mu\left(\{a_1,\dots,a_k\}\right) \leq \nu\Bigg(\bigcup_{j=1}^k \overline{B_{\varepsilon_A}(a_j)}\Bigg)+\varepsilon_A
	$$
	for every $0 < \delta < \varepsilon_A$, which proves \eqref{min formula}.
\end{proof}

Note that the reason why we excluded the case when $\mu = \nu$ in \eqref{min formula} is that then for every $A\subseteq S_\mu$ and $\varepsilon>0$ we have $\mu(A)\leq \nu(\overline{A^{\varepsilon}})+\varepsilon$, and for every $\emptyset\neq A\subseteq S_\mu$ we have $\mu(A) > 0 = \nu(\emptyset)+0 = \nu(\overline{A^{0}})+0$.
Of course if we had defined $\overline{A^{0}}$ to be $A$, then \eqref{min formula} with $\varepsilon\geq 0$ instead of $\varepsilon>0$ would hold for the case $\mu = \nu$ as well.
However, we prefer not to change usual notations.

The following proposition plays a key role in the proof. 
But before stating it we introduce some notations.
The convex hull of two points $x$ and $y$ will be denoted by $[x,y]$, and the symbol $]x,y[$ will stand for the set $[x,y]\setminus\{x,y\}$. If $f$ is a real valued function on $X$ and $c\in\mathbb{R}$, then the sets 
$\{x\in X\,|\,f(x)<c\}$, $\{x\in X\,|\,f(x)=c\}$, and $\{x\in X\,|\,f(x)\leq c\}$
will be denoted by $\{f<c\}$, $\{f=c\}$, and $\{f\leq c\}$, respectively.

\begin{proposition}\label{burok} 
Let $(X,\|\cdot\|)$ be a separable real Banach space, $\mu\in\Fx\setminus\Delta_X$ and $K$ be the convex hull of $\spm$.
Assume that $\xk$ is a vertex of $K$ (which is a polytope) and set $\lk := \mu(\{\xk\})$ (for which we obviously have $0<\lk<1$). 
Then for every $\vartheta\in\me$ with $S_{\vartheta}\subseteq K$ the following two conditions are equivalent:
\begin{itemize}
\item[(i)] $\vartheta=\lk\delta_{\xk}+(1-\lk)\widetilde{\vartheta}$ where $\widetilde{\vartheta}\in\me$ with $S_{\widetilde{\vartheta}}\subseteq K\setminus B_r(\xk)$ for some $r>0$,
\item[(ii)] there exist a number $0 < \rho \leq 1-\lk$ and a half-line $\mathfrak{e}$ starting from $\xk$ such that the restriction $W_{\vartheta}|_{\mathfrak{e}}$ is of the following form:
\begin{align}\label{eq:tanufv}
	W_{\vartheta}|_{\mathfrak{e}}(x)=\left\{
        \begin{array}{ll}
            1 &\quad\mbox{if} \quad \|x-\xk\|\geq 1,\\
            \|x-\xk\| &\quad\mbox{if} \quad 1-\lk<\|x-\xk\|<1,\\
            1-\lk &\quad\mbox{if}\quad 1-\lk-\rho\leq\|x-\xk\|\leq 1-\lk.
        \end{array}    
    \right.
    \end{align}
    \end{itemize}
Moreover, $S_{\varphi(\mu)} \subseteq K$ and $\xk$ is an isolated atom of $\varphi(\mu)$ with $(\varphi(\mu))(\{\xk\}) = \lk$.
\end{proposition}

\begin{figure}[H]
	\begin{tikzpicture}[scale=1]
	\node at (-3,-2) {$Y$};
	\draw [fill] (0,0) circle [radius=1pt];
	\node [above] at (0,0) {$\check{x}$};
	\draw [thin] (0.1987,-0.8013) -- (0.8013,-0.1987);
	\draw [thin] plot [smooth] coordinates {(0.8,-0.2) (1,0.3) (0.8,0.8) (0.3,1) (-0.2,0.8)};
	\draw [thin] (-0.1987,0.8013) -- (-0.8013,0.1987);
	\draw [thin] plot [smooth] coordinates {(-0.8,0.2) (-1,-0.3) (-0.8,-0.8) (-0.3,-1) (0.2,-0.8)};
	\draw [ultra thin] (0,0) -- (-1,-0.3);
	\node at (-0.6,-0.03) {1};
	\draw (0.5,-0.5) -- (-2,2);
	\node [right] at (-1.5,1.5) {$\mathfrak{e}$};
	\draw[fill, lightgray] (0.5,-0.5) -- (0,-2) -- (1,-3) -- (3,-3) -- (4,-1) -- (3,0) -- (0.5,-0.5);
	\node at (1.2,-2.25) {$K$};
	\draw[fill, gray] (2,-2) circle [radius=10pt];
	\draw[fill, gray] (2,-2) -- (2.5,-1) -- (3.5,-1.5);
	\draw [fill, ultra thick, gray] plot [smooth] coordinates {(3,0) (2,-1) (4,-1)};
	\draw [fill, ultra thick, gray] plot [smooth] coordinates {(0,-2) (2,-1) (1,-1) (0.5,-2.5)};
	\draw[fill] (0,-2) circle [radius=1pt];	
	\draw[fill] (1,-3) circle [radius=1pt];	
	\draw[fill] (3,-3) circle [radius=1pt];	
	\draw[fill] (4,-1) circle [radius=1pt];	
	\draw[fill] (3,0) circle [radius=1pt];
	\draw[fill] (1.5,-1) circle [radius=1pt];
	\draw[fill] (3,-2) circle [radius=1pt];
	\draw[fill] (3,-0.5) circle [radius=1pt];
	\draw[fill] (2.5,-3) circle [radius=1pt];
	\node [white] at (2.6,-1.4) {$S_{\widetilde\vartheta}$};
	\draw[fill] (0.5,-0.5) circle [radius=1pt];
	\node [above] at (0.5,-0.5) {$\xk$};
	\draw [dashed] (3,2) -- (-2,-3);
	\node [above] at (2.7,2) {$\{f = \hat{c}\}$};
	\draw [dashed] (3.25,1.75) -- (-1.75,-3.25);
	\node [right] at (3.2,1.9) {$\{f = c\}$};
	\end{tikzpicture}
	\caption{Illustration for Proposition \ref{burok} on the finite dimensional subspace $Y$. The support $S_\mu$ consists of the set of black points in $K$.}
\end{figure}
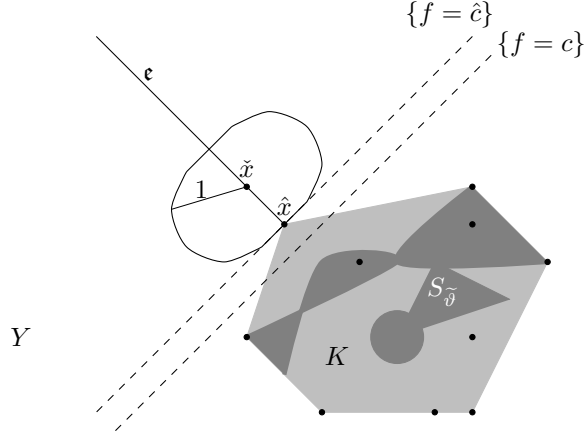

\begin{proof} 
	First, we construct a half-line $\mathfrak{e}$ which starts from $\xk$ and satisfies 
	\begin{align}\label{eq:xK tav}
	\underline{\mathrm{d}}\left(\{x\},K\right) = \|x-\xk\| < \|x-k\| \qquad (\forall\; x\in\mathfrak{e}, k\in K\setminus{\{\xk\}}).
	\end{align}
	Being the convex hull of a finite set, each vertex of $K$ is strongly exposed, i.e. there exists a continuous linear functional $f\in X^{\ast}$ with 
	$$
	\ck := \max\{f(y)\,|\, y\in K\} = f(\xk) \quad \text{and} \quad K\setminus \{\xk\} \subset \{f<\ck\}.
	$$
	Let $Y$ be the subspace generated by $K$.	
	We fix an $\xv\in Y$ such that $\xk\in\overline{B_1(\xv)}$ and
	$$
	\overline{B_1(\xv)} \cap \{f\leq\ck\} \cap Y \subseteq \{f=\ck\}.
	$$
	Note that as $Y$ is finite dimensional, the existence of such an $\xv$ is guaranteed.
	Now, we define $\mathfrak{e}$ to be the half-line starting from $\xk$ and going through $\xv$.
	It is straightforward that $\mathfrak{e}$ fulfils \eqref{eq:xK tav}.
	
	Next, we consider an arbitrary $\vartheta\in\me$ which satisfies (i). It is clear form the compactness of $S_\mu$ and $S_\vartheta$, and the isolatedness of the point $\xk$ in both $S_\mu$ and $S_\vartheta$, that there is a number $c<\ck$ such that
	\begin{align}\label{eq:jo c}
	(\spm\cup S_{\vartheta}) \setminus \{\xk\}\subseteq\{f<c\}
	\end{align}
	holds.
	Consequently, for every $x\in\mathfrak{e}$ we have 
	$$
	\underline{\mathrm{d}}\left(\{x\},K\right) = \|x-\xk\| > \underline{\mathrm{d}}\left(\{x\}, \{f\leq c\}\cap Y\right) > \underline{\mathrm{d}}\left(\{x\}, (\spm\cup S_{\vartheta}) \setminus \{\xk\}\right).
	$$
	Therefore, if $x\in\mathfrak{e}$ and $\alpha := \|x-\xk\| > 0$, then there exists a $y\in\;]x,\xk[\;\subseteq\mathfrak{e}$ which satisfies the following equations:
	\begin{align}\label{eq:eltolas2}
	\lk = \vartheta(\{\xk\}) = \vartheta(\overline{B_{\alpha}(x)}) = \vartheta(\overline{B_{\alpha}(z)}) \qquad (\forall\; z\in[x,y]).
	\end{align}
	and
	\begin{align}\label{eq:eltolas}
	\lk = \mu(\{\xk\}) = \mu(\overline{B_{\alpha}(x)}) = \mu(\overline{B_{\alpha}(z)}) \qquad (\forall\; z\in[x,y]).
	\end{align}
	In fact, $y$ can be chosen to be any point on $]x,\xk[$ such that 
	$$
	\|x-y\| \leq \underline{\mathrm{d}}\left(\{f=\ck\},\{f=c\}\right).
	$$
		
	We proceed to prove the equivalence of (i) and (ii). 
	Trivially, both (i) and (ii) implies that $\vartheta \notin \Delta_X$, and therefore from now on we may and do assume that $\vartheta$ is not a Dirac measure.
	Recall that according to Proposition \ref{min} we have
	\begin{align}\label{eq:p(x,mu)}
	W_\vartheta(x) = \min\left\{\varepsilon > 0\,|\,1\leq\vartheta(\overline{B_{\varepsilon}(x)})+\varepsilon\right\} \qquad (\forall\;x\in X).
	\end{align}
	If (i) holds, then  by combining \eqref{eq:p(x,mu)} with \eqref{eq:xK tav} and \eqref{eq:eltolas2} we obtain that $W_{\vartheta}|_{\mathfrak{e}}$ is of the form \eqref{eq:tanufv} with $\rho := \min\left(1-\lk,\underline{\mathrm{d}}\left(\{f=\ck\},\{f=c\}\right)\right)$.
	
	Conversely, we suppose that $\vartheta\in\me$, $S_{\vartheta}\subseteq K$ and $\vartheta$ satisfies (ii). 
	Let $x_1$ and $x_2$ be the points on $\mathfrak{e}$ which satisfy $\|x_1-\xk\| = 1-\lk$ and $\|x_2-\xk\| = 1-\lk-\rho$.
	By \eqref{eq:tanufv} we have $W_\vartheta(x_1) = W_\vartheta(x_2) = 1-\lk$.
	Therefore on one hand, we obtain
	$$
	1\leq \vartheta(\overline{B_{1-\lk}(x_1)}) + 1-\lk = \vartheta(\{\xk\}) + 1-\lk,
	$$
	from which $\lk \leq \vartheta(\{\xk\})$ follows.
	On the other hand, 
	$$
	1 > \vartheta(\overline{B_{1-\lk-\delta}(x_2)}) + 1-\lk-\delta \geq \vartheta(\{\xk\}) + 1-\lk-\delta \quad (\forall\; 0<\delta< \rho)
	$$
	is satisfied. Hence we infer $\vartheta(\{\xk\}) < \lk+\delta$ for every $\delta > 0$, and thus trivially $\vartheta(\{\xk\}) = \lk$ holds.
	But we also observe the following:
	$$
	1 > \vartheta\left(\overline{B_{1-\lk-\delta}(x_2)}\right) + 1-\lk-\delta \geq \vartheta\left(\overline{B_{1-\lk-\tfrac{\rho}{2}}(x_2)}\right) + 1-\lk-\delta \quad \left(\forall\; 0<\delta<\tfrac{\rho}{2}\right)
	$$
	which implies 
	$$
	\lk \geq \vartheta\left(\overline{B_{1-\lk-\tfrac{\rho}{2}}(x_2)}\right) \geq \vartheta\left(\{\xk\}\right) = \lk,
	$$
	whence we conclude that $\xk$ is indeed an isolated atom of $\vartheta$.
	
	For the last statement first, by Proposition \ref{1-tav} we infer $\pi(\delta_x,\mu) = 1$ for every $x\notin K^1$. 
	The $\varphi$-invariance of the witness function gives 
	$$
	1=\pi(\delta_x,\varphi(\mu)) \qquad (x\in X\setminus K^1),
	$$ 
	and hence, again by Proposition \ref{1-tav}, we conclude
	$$
	S_{\varphi(\mu)}\cap B_1(x)=\emptyset \qquad (x\in X\setminus K^1).
	$$ 
	Consequently, we obtain
	$$
	S_{\varphi(\mu)}\subseteq X\setminus(X\setminus K^1)^1\subseteq K.
	$$
	Finally, an application of the equivalence of (i) and (ii) gives the rest.
\end{proof}

We have the following consequence.

\begin{corollary}\label{corollary ketpont} 
Let $(X,\|\cdot\|)$ be a separable real Banach space. 
If $\mu\in\Fx$ such that $\# S_\mu \leq 2$, then $\varphi(\mu)=\mu$. 
Moreover, if $\nu\in\me$ with $W_\mu \equiv W_\nu$, then $\mu$ and $\nu$ coincide.
\end{corollary}

According to the above results, now we know that $\varphi$ fixes every measure which has at most two points in its support. 
Although we are expecting the same for all $\mu\in\Fx$, right now we only have some information about the behaviour of $\varphi(\mu)$ near to the vertices of the convex hull of its support.


\subsection{Third major step: the story beyond vertices}

Here we show a procedure how the behaviour of $\varphi(\mu)$ can be completely explored in case when $\mu$ is a finitely supported measure.
In order to do so, we need to introduce some technical notations.
Let $s>0$ be a positive parameter and define the \emph{$s$-L{\'e}vy--Prokhorov distance} $\pi_s \colon \me\times\me\to[0,1]$ by the following formula:
\begin{align}\label{pi_s defje}
\pi_s(\mu,\nu) := \inf\left\{\varepsilon>0\,|\,\forall\; A\in\bx\colon s\cdot\mu(A)\leq s\cdot\nu(A^{\varepsilon})+\varepsilon\right\} \;\; (\forall\;\mu,\nu\in\me).
\end{align}
Note that although we do not know at this point whether $\pi_s$ defines a metric on $\me$, we will see this later.
The \emph{$s$-witness function} (or modified witness function) of $\mu\in\me$ is defined by
\begin{align*}
W_{s,\mu}\colon X\to\mathbb{R}, \quad W_{s,\mu}(x):=\pi_s(\delta_x,\mu).
\end{align*}
Obviously, if we set $s=1$, then we get the original L{\'e}vy--Prokhorov metric and witness function.

In the next two lemmas we collect some properties of the $s$-L{\'e}vy--Prokhorov distance analogous to those provided in the previous major step.

\begin{lemma}\label{s-min}
Let $(X,\|\cdot\|)$ be a separable real Banach space and $s>0$. 
Then $(\me,\pi_s)$ is a metric space.
Furthermore, for every $\mu\in\Fx$ and $\nu\in\me$ with $\nu\neq\mu$ we have
$$
\pi_s(\mu,\nu) = \min\left\{\varepsilon > 0\,\big|\,\forall\; A\subseteq S_\mu \colon s\cdot\mu(A)\leq s\cdot\nu(\overline{A^{\varepsilon}})+\varepsilon\right\}.
$$
\end{lemma}

\begin{proof}
For the sake of clarity, let us use more detailed notations here. 
If $\|\cdot\|$ is a norm on $X$, then denote by $A^{\varepsilon,\|\cdot\|}$ and $\pi_{s,\|\cdot\|}$ the open $\varepsilon$-neighborhood of $A$ and the $s$-L{\'e}vy--Prokhorov metric with respect to $\|\cdot\|$, respectively.
Observe that the Borel $\sigma$-algebras of $(X,\|\cdot\|)$ and $\left(X,\frac{1}{s}\|\cdot\|\right)$ coincide as the norms are equivalent. 
By an elementary computation we have $A^{s\delta,\|\cdot\|} = A^{\delta,\frac{1}{s}\|\cdot\|}$, which yields

\begin{equation}\label{s-pi}
\begin{split}
\pi_{s,\|\cdot\|}(\mu,\nu)& = \inf\left\{\varepsilon>0\,\Big|\,\forall\; A\in\bx\colon s\cdot\mu(A)\leq s\cdot\nu\big(A^{\varepsilon,\|\cdot\|}\big)+\varepsilon\right\}\\
& = \inf\left\{s\delta>0\,\Big|\,\forall\; A\in\bx\colon s\cdot\mu(A)\leq s\cdot\nu\big(A^{s\delta,\|\cdot\|}\big)+s\delta\right\}\\
& = s\cdot \inf\left\{\delta>0\,\Big|\,\forall\; A\in\bx\colon \mu(A)\leq \nu\big(A^{s\delta,\|\cdot\|}\big)+\delta\right\}\\
& = s\cdot \inf\left\{\delta>0\,\Big|\,\forall\; A\in\bx\colon \mu(A)\leq \nu\big(A^{\delta,\frac{1}{s}\|\cdot\|}\big)+\delta\right\}\\
& = s\cdot \pi_{\frac{1}{s}\|\cdot\|}(\mu,\nu)
\end{split}
\end{equation}
for every $\mu,\nu\in\me$.
In particular, $\pi_s$ is a metric on $\me$, and using the formula \eqref{min formula} completes the proof.
\end{proof}

We omit the proof of the following lemma as it is a straightforward consequence of \eqref{s-pi}.

\begin{lemma}\label{pi_s gyujtolemma}
	Let $(X,\|\cdot\|)$ be a separable real Banach space, $\mu\in\Fx\setminus\Delta_X$ and $s>0$. 
	Let us denote the convex hull of $S_\mu$ by $K$, and assume that $\xk$ is a vertex of $K$.
	Set $\lk := \mu(\{\xk\}) \in (0,1)$.
	Then for every $\vartheta\in\me$ with $S_\vartheta\subseteq K$ the following two conditions are equivalent:
	\begin{itemize}
		\item[(i)] $\vartheta = \lk\delta_{\xk}+(1-\lk)\widetilde{\vartheta}$ where $\widetilde{\vartheta}\in\me$ with $S_{\widetilde{\vartheta}}\subseteq K\setminus B_r(\xk)$ for some $r>0$,
		\item[(ii)] there exist a number $0 < \rho \leq s(1-\lk)$ and a half-line $\mathfrak{e}$ starting from $\xk$ such that $W_{s,\vartheta}|_{\mathfrak{e}}$ has the following form:
			$$W_{s,\vartheta}|_{\mathfrak{e}}(x)=\left\{
			\begin{array}{ll}
			s &\quad\mbox{if} \quad \|x-\xk\|\geq s,\\
			\|x-\xk\| &\quad\mbox{if} \quad s(1-\lk)<\|x-\xk\|<s,\\
			s(1-\lk) &\quad\mbox{if}\quad s(1-\lk)-\rho\leq\|x-\xk\|\leq s(1-\lk).
			\end{array}    
			\right.$$
	\end{itemize}
	As a consequence we have that if $\mu\in\Fx, \#S_\mu\leq 2$, $\nu\in\me$ and $W_{s,\mu} \equiv W_{s,\nu}$, then $\mu=\nu$.
\end{lemma}

Next, let us suppose for a moment that $m\in\mathbb{N}$ pieces of atoms of $\vartheta\in\me$ have been already detected.
(For instance by Lemma \ref{burok}, if $\vartheta = \varphi(\mu)$ with $\mu\in\Fx$ then the atoms of $\vartheta$ in the vertices of the convex hull of $S_\vartheta$ can be detected.)
Our aim with the forthcoming lemma is to describe a modified witness function of the remaining part of $\vartheta$ in terms of the (original) L{\'e}vy--Prokhorov distances between $\vartheta$ and some measures which are supported on at most $m+1$ points.
This will be later utilised in order to explore the action of $\varphi$ on $\Fx$.

\begin{lemma}\label{tanu} 
Let $(X,\|\cdot\|)$ be a separable real Banach space and $\vartheta\in\me$. 
Let $x\in X$ and $\{y_{j,l}\,|\,1\leq j\leq k, 1\leq l \leq d_j\}\subset X$ be some pairwise different points such that
$$
\rho_j := \|x-y_{j,1}\| = \|x-y_{j,l}\| \qquad (\forall\; 1\leq l \leq d_j)
$$
holds for every $1\leq j\leq k$, 
$$
\rho_j > \rho_{j+1} > 0 \qquad (\forall\; 1\leq j\leq k-1),
$$ 
and 
$$
w_{j,l} := \vartheta(\{y_{j,l}\}) > 0 \qquad (\forall\; 1\leq j\leq k, 1\leq l \leq d_j).
$$
We also set 
$$
w_j := \sum_{l=1}^{d_j} w_{j,l} = \vartheta(\{y_{j,1},\dots,y_{j,d_j}\}) \qquad (\forall\; 1\leq j\leq k),
$$
$$
\widetilde{w} := 1-\sum_{j=1}^kw_j
$$
and 
\begin{align}\label{def:eta}
\eta_{r}:=\sum_{j=1}^{r}\sum_{l=1}^{d_j}w_{j,l}\cdot\delta_{y_{j,l}} + \bigg(1-\sum_{j=1}^{r} w_j\bigg)\cdot\delta_x\in\Fx \qquad (\forall\; 0\leq r\leq k).
\end{align}
Furthermore, denote by $\widetilde{\vartheta}\in\me$ the measure which satisfies
\begin{align}\label{def:teta}
\vartheta=\sum_{j=1}^k\sum_{l=1}^{d_j}w_{j,l}\cdot\delta_{y_{j,l}}+\widetilde{w}\cdot\widetilde{\vartheta}.
\end{align}

Then the $\widetilde{w}$-witness function of $\widetilde{\vartheta}$ can be expressed in terms of the L{\'e}vy--Prokhorov distances of $\vartheta$ and $\eta_r$'s in the following way:
\begin{align}\label{eq:folemma}
W_{\widetilde{w},\widetilde{\vartheta}}(x) = 
	\left\{
        \begin{array}{ll}
            \pi(\delta_x,\vartheta) & \; \mbox{\textrm{if }} x\mbox{ is not }(P_1) \\
            \pi(\eta_r,\vartheta) &\; \mbox{\textrm{if }} x\mbox{ is }(P_r)\mbox{ but not }(P_{r+1})\mbox{ with some }1\leq r < k\\
            \pi(\eta_k,\vartheta) & \; \mbox{\textrm{if }} x\mbox{ is }(P_k)
        \end{array}    
    \right.
\end{align}
where for every $1\leq r\leq k$ the property $(P_r)$ means
\begin{align}\label{Pr}\tag{$P_r$}
\pi(\eta_{r-1},\vartheta)\leq\rho_r.
\end{align}
\end{lemma}

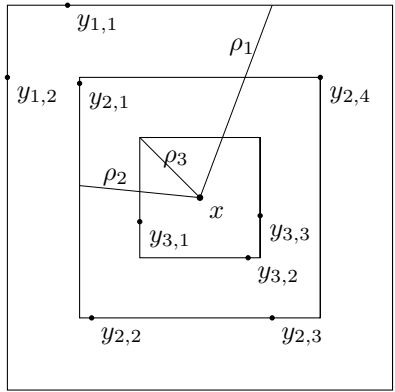
\begin{figure}[H]
	\begin{tikzpicture}[scale=0.8]
	\draw[fill] (0,0) circle [radius=1.3pt];
	\node [below right, black] at (0,0) {$x$};
	\draw [thin] (1,1) -- (1,-1) -- (-1,-1) -- (-1,1) -- (1,1) -- (1,-1);
	\draw [thin] (2,2) -- (2,-2) -- (-2,-2) -- (-2,2) -- (2,2) -- (2,-2);
	\draw [thin] (3.2,3.2) -- (3.2,-3.2) -- (-3.2,-3.2) -- (-3.2,3.2) -- (3.2,3.2) -- (3.2,-3.2);
	\draw [ultra thin] (0,0) -- (-1,1);
	\node [black] at (-0.4,0.65) {$\rho_3$};
	\draw [ultra thin] (0,0) -- (-2,0.2);
	\node [black] at (-1.4,0.35) {$\rho_2$};	
	\draw [ultra thin] (0,0) -- (1.2,3.2);
	\node [black] at (0.7,2.5) {$\rho_1$};
	\draw[fill] (1,-0.3) circle [radius=1pt];
	\node [below right, black] at (1,-0.3) {$y_{3,3}$};
	\draw[fill] (0.8,-1) circle [radius=1pt];
	\node [below right, black] at (0.8,-1) {$y_{3,2}$};
	\draw[fill] (-1,-0.4) circle [radius=1pt];
	\node [below right, black] at (-1,-0.4) {$y_{3,1}$};
	\draw[fill] (2,2) circle [radius=1pt];
	\node [below right, black] at (2,2) {$y_{2,4}$};
	\draw[fill] (1.2,-2) circle [radius=1pt];
	\node [below right, black] at (1.2,-2) {$y_{2,3}$};
	\draw[fill] (-1.8,-2) circle [radius=1pt];
	\node [below right, black] at (-1.8,-2) {$y_{2,2}$};	
	\draw[fill] (-2,1.9) circle [radius=1pt];
	\node [below right, black] at (-2,1.9) {$y_{2,1}$};
	\draw[fill] (-3.2,2) circle [radius=1pt];
	\node [below right, black] at (-3.2,2) {$y_{1,2}$};
	\draw[fill] (-2.2,3.2) circle [radius=1pt];
	\node [below right, black] at (-2.2,3.2) {$y_{1,1}$};
	\end{tikzpicture}
	\caption{An illustration when $X=\mathbb{R}^2$ with the $\ell^\infty$-norm.}
\end{figure}

\begin{remark}
	It is extremely important to observe that the subscripts in the lemma above highly depend on the actual position of $x$. 
	For instance on Figure 1 with that particular $x$ we have $k=3$. 
	However, if $x$ is moved slightly to the right, then $k$ becomes 7. 
	In particular, this changes \eqref{Pr} and therefore \eqref{eq:folemma} as well.
\end{remark}

\begin{proof}[Proof of Lemma \ref{tanu}]
We split our proof into five parts.

\emph{Part 1.} First, we prove that for each $1\leq r \leq k$ and $0 < \varepsilon < \rho_r$ we have 
\begin{align}\label{ineq:A}
\eta_r(A) \leq \vartheta(\overline{A^\varepsilon}) + \varepsilon \quad (\forall\;A\subseteq S_{\eta_r})
\end{align}
if and only if
\begin{align}\label{ineq:singlex}
\eta_r(\{x\}) \leq \vartheta(\overline{B_\varepsilon(x)}) + \varepsilon
\end{align}
is satisfied. 
One direction is obvious.
In order to see the reverse implication, observe that \eqref{ineq:A} holds trivially if $x\notin A$. On the other hand, if $x\in A$, then \eqref{ineq:singlex} yields \eqref{ineq:A} for this $A$ by the following estimation:
$$
\eta_r(A) = \eta_r(\{x\}) + \eta_r(A\setminus\{x\}) \leq \vartheta(\overline{B_\varepsilon(x)}) + \varepsilon + \vartheta(A\setminus\{x\}) \leq \vartheta(\overline{A^\varepsilon})+\varepsilon.
$$

\emph{Part 2.} 
Here we show that the right-hand side of \eqref{eq:folemma} is well defined. 
First, we observe that by Proposition \ref{min} $x$ is \eqref{Pr} if and only if
\begin{align}\label{ineq:eta r-1}
\eta_{r-1}(A)\leq\vartheta(\overline{A^{\rho_r}})+\rho_r \quad (\forall\;A\subseteq S_{\eta_{r-1}}).
\end{align}
But by Part 1, this is equivalent to the following inequality:
\begin{align}\label{Pr'}\tag{$P_r'$}
1-\sum_{j=1}^{r-1}w_j\leq\vartheta(\overline{B_{\rho_r}(x)})+\rho_r.
\end{align}

Next, let $2\leq r\leq k$. In order to see the well-definedness, it is enough to show that if $x$ is $(P_r)$, then $x$ is also $(P_{r-1})$. 
So assume that $x$ is $(P_r)$. 
Since $\rho_r<\rho_{r-1}$ and 
$$
w_{r-1} = \sum_{l=1}^{d_{r-1}}w_{r-1,l} = \vartheta(\{y_{r-1,1},\dots y_{r-1,d_{r-1}}\}) \leq\vartheta\left(\overline{B_{\rho_{r-1}}(x)}\setminus \overline{B_{\rho_{r}}(x)}\right),
$$
we obtain
\begin{align}\label{Pr-1'}\tag{$P_{r-1}'$}
1-\sum_{j=1}^{r-2}w_j \leq \vartheta(\overline{B_{\rho_{r-1}}(x)})+\rho_{r-1}.
\end{align}
Therefore $x$ is indeed $(P_{r-1})$.

\emph{Part 3.}
Next, we verify \eqref{eq:folemma} in case when $x$ is not $(P_1)$, i.e. $\pi(\delta_x,\vartheta)>\rho_1$. 
Observe that since
$$
1 > \vartheta(\overline{B_{\rho_1}(x)})+\rho_1 = \widetilde{w} \cdot \widetilde{\vartheta}(\overline{B_{\rho_1}(x)})+\sum_{i=1}^kw_i+\rho_1,
$$
we have
$$
\widetilde{w} > \widetilde{w}\cdot\widetilde{\vartheta}(\overline{B_{\rho_1}(x)})+\rho_1,
$$
and thus 
$$
\pi_{\widetilde{w}}(\delta_x,\widetilde{\vartheta})>\rho_1
$$ 
follows.
Using this fact we obtain
\begin{equation*}
\begin{split}
\pi_{\widetilde{w}}(\delta_x,\widetilde{\vartheta})& = \min\left\{\varepsilon>\rho_1\,\big|\,\widetilde{w}\leq\widetilde{w}\cdot\widetilde{\vartheta}(\overline{B_{\varepsilon}(x)})+\varepsilon\right\}\\
& = \min\left\{\varepsilon>\rho_1\,\bigg|\,1\leq\sum_{j=1}^kw_j+\widetilde{w}\cdot\widetilde{\vartheta}(\overline{B_{\varepsilon}(x)})+\varepsilon\right\}\\
& = \min\left\{\varepsilon>\rho_1\,\big|\,1\leq\vartheta(\overline{B_{\varepsilon}(x)})+\varepsilon\right\}\\
& = \min\left\{\varepsilon>0\,\big|\,1\leq\vartheta(\overline{B_{\varepsilon}(x)})+\varepsilon\right\} = \pi(\delta_x,\vartheta),
\end{split}
\end{equation*}
which completes this part.

\emph{Part 4.} 
We proceed to show \eqref{eq:folemma} in the case when $x$ is $(P_r)$ but not $(P_{r+1})$ with some $1\leq r < k$.
As in the previous part, first we estimate the value of $\pi_{\wh}(\delta_x,\tetah)$. 
According to the re-phrasing \eqref{Pr'} and the assumption, we have
\begin{align*}
1-\sum_{j=1}^{r-1}w_j\leq\vartheta(\overline{B_{\rho_r}(x)})+\rho_r =
\sum_{i=r}^kw_i+\widetilde{w}\cdot\widetilde{\vartheta}(\overline{B_{\rho_r}(x)})+\rho_r
\end{align*}
and
\begin{align*}
1-\sum_{j=1}^{r}w_j>\vartheta(\overline{B_{\rho_{r+1}}(x)})+\rho_{r+1}=\sum_{i=r+1}^kw_i+\widetilde{w}\cdot\widetilde{\vartheta}(\overline{B_{\rho_{r+1}}(x)})+\rho_{r+1}.
\end{align*}
Observe that these inequalities are equivalent to
\begin{align*}
\widetilde{w}\leq\widetilde{w}\cdot\widetilde{\vartheta}(\overline{B_{\rho_r}(x)})+\rho_r
\end{align*}
and
\begin{align*}
\widetilde{w}>\widetilde{w}\cdot\widetilde{\vartheta}(\overline{B_{\rho_{r+1}}(x)})+\rho_{r+1},
\end{align*}
respectively.
Thus we conclude that
\begin{align*}
\rho_{r+1}<\pi_{\widetilde{w}}(\delta_x,\widetilde{\vartheta})\leq\rho_r.
\end{align*}
In particular, $\widetilde{\vartheta}$ is different from $\delta_x$, and we have
\begin{align}
\label{eq:x}\pi_{\ww}(\delta_x,\tteta) = \min\left\{\rho_{r+1}<\varepsilon\leq\rho_r \,\big|\, \ww\leq\ww\cdot\tteta(\bex)+\varepsilon\right\}.
\end{align}

From now on we consider two cases: (a) when $\rho_{r+1} < \pi_{\wh}(\delta_x,\tetah)<\rho_r$, and (b) when $\pi_{\wh}(\delta_x,\tetah)=\rho_r$. 
Assume first that (a) is fulfilled.
Then \eqref{eq:x} becomes
\begin{equation*}
\begin{split}
\pi_{\ww}\left(\delta_x,\tteta\right)& = \min\left\{\rho_{r+1}<\varepsilon<\rho_r \,\Big|\, 1-\sum_{i=1}^rw_i\leq \sum_{i=r+1}^k w_i+\ww\cdot\tteta(\bex)+\varepsilon\right\}\\
& = \min\left\{\rho_{r+1}<\varepsilon<\rho_r \,\Big|\, 1-\sum_{i=1}^r w_i\leq \vartheta(\bex)+\varepsilon\right\}\\
& = \min\left\{\rho_{r+1}<\varepsilon<\rho_r\,\big|\,\eta_r(\{x\})\leq\vartheta(\overline{\{x\}^{\varepsilon}})+\varepsilon\right\}\\
& = \min\left\{\rho_{r+1}<\varepsilon<\rho_r \,\big|\,\forall\; A\subseteq S_{\eta_r}\colon \eta_r(A)\leq \vartheta(\overline{A^{\varepsilon}})+\varepsilon\right\} \quad \text{(by Part 1)}\\
& = \min\left\{\varepsilon > 0\,|\,\forall\; A\subseteq S_{\eta_r}\colon \eta_r(A)\leq \vartheta(\overline{A^{\varepsilon}})+\varepsilon\right\}\\
& = \pi(\eta_r,\vartheta),
\end{split}
\end{equation*}
which is exactly the desired equation. 
Second, suppose that (b) is satisfied. 
Consequently, we have 
$$
1-\sum_{j=1}^k w_j=\ww>\ww\cdot\tteta(\bex)+\varepsilon \quad (\forall\;\rho_{r+1}<\varepsilon<\rho_r),
$$
whence
$$
\eta_r(\{x\}) = 1-\sum_{i=1}^r w_i>\sum_{i=r+1}^k w_k+\ww\cdot\tteta(\bex)+\varepsilon=\vartheta(\bex)+\varepsilon
$$
follows for every $\rho_{r+1}<\varepsilon<\rho_r$. 
In particular, we get 
$$
\pi(\eta_r,\vartheta)\geq\rho_r.
$$
Finally, we verify that the converse inequality holds as well. 
Suppose indirectly that there exists an $A\subseteq S_{\eta_r}$ such that
$$
\eta_r(A)>\vartheta(\overline{A^{\rho_r}})+\rho_r.
$$
Clearly, $x\notin A$ contradicts the above inequality, thus $x\in A$ follows.
Therefore we have
\begin{align*}
\eta_r(A)>\vartheta(\overline{A^{\rho_r}})+\rho_r&\geq \vartheta(A\cup\overline{B_{\rho_r}(x)})+\rho_r\\
&\geq\vartheta(A
\setminus\{x,y_{r,1},\dots,y_{r,d_r}\})+\vartheta(\brx)+\rho_r\\
&=\eta_r(A\setminus\{x,y_{r,1},\dots,y_{r,d_r}\})+\vartheta(\overline{B_{\rho_r}(x)})+\rho_r.
\end{align*}
Consequently,
\begin{align}\label{hasonlolesz}
\begin{split}
1-\sum_{j=1}^{r-1}w_j & = \eta_r(\{x,y_{r,1},\dots,y_{r,d_r}\}) \\
& \geq \eta_r(A) - \eta_r(A\setminus\{x,y_{r,1},\dots,y_{r,d_r}\}) >\vartheta(\brx)+\rho_r,
\end{split}
\end{align}
which contradicts \eqref{Pr}. 
This completes the present part. 

\emph{Part 5.} Finally, we prove \eqref{eq:folemma} when $x$ is $(P_k)$, i.e. $\pi(\eta_{k-1},\vartheta)\leq\rho_k$. 
We have to show that $\pi(\eta_{k},\vartheta)=\pi_{\ww}(\delta_x,\tteta)$. 
Because of the assumption, we have
$$
1-\sum_{j=1}^{k-1}w_j\leq\vartheta(\bkx)+\rho_k=w_k+\ww\cdot\tteta(\bkx)+\rho_k
$$
which implies $\ww\leq\ww\cdot\tteta(\bkx)+\rho_k$, and hence, 
$$
\pi_{\ww}(\delta_x,\tteta)\leq\rho_k.
$$ 
We consider three cases: (a) when $0<\pi_{\ww}(\delta_x,\tteta)<\rho_k$, (b) when $\pi_{\ww}(\delta_x,\tteta)=\rho_k$, and (c) when $\pi_{\ww}(\delta_x,\tteta)=0$. 
First, let us suppose (a).
In this case we have
\begin{equation*}
\begin{split}
\pi_{\ww}(\delta_x,\tteta)& = \min\left\{0<\varepsilon<\rho_k\,\big|\,\ww\leq\ww\cdot\tteta(\bex)+\varepsilon\right\}\\
& = \min\left\{0<\varepsilon<\rho_k\,\big|\,\eta_k(\{x\}) \leq\vartheta(\bex)+\varepsilon\right\}\\
& = \min\left\{0<\varepsilon<\rho_k\,|\,\forall\; A\subseteq S_{\eta_{k}}\colon \eta_{k}(A)\leq\vartheta(\overline{A^{\varepsilon}})+\varepsilon\right\} \quad \text{(by Part 1)}\\
& = \pi(\eta_{k},\vartheta).
\end{split}
\end{equation*}
Second, we assume (b). 
Let us observe the following for every $\varepsilon<\rho_k$:
$$
\eta_{k}(\{x\})=\ww>\ww\cdot\tteta(\bex)+\varepsilon=\vartheta(\bex)+\varepsilon,
$$
which implies $\pi(\eta_{k},\vartheta)\geq\rho_k$. To show the converse inequality, i.e. $\pi(\eta_k,\vartheta)\leq\rho_k$, assume indirectly that there exists an $A\subseteq S_{\eta_{k}}$ such that
$$\eta_{k}(A)>\vartheta(\overline{A^{\rho_k}})+\rho_k.$$
Very similarly, as in the verification of \eqref{hasonlolesz}, we conclude that this inequality contradicts $(P_k)$.
Finally, the case (c) is trivial.
\end{proof}

Since the modified witness function is obviously continuous, we also know the value of $W_{\ww,\tteta}(x)$ when $x\in\left\{y_{j,l}\,|\, 1\leq j\leq k, 1\leq l\leq d_j\right\}$.
Therefore if $m$ pieces of atoms of $\vartheta\in\me$ have been already detected, then a modified witness function of the remaining part of $\vartheta$ can be calculated in terms of the L{\'e}vy--Prokhorov distances between $\vartheta$ and some measures supported on a set of at most $m+1$ points.


\subsection{Final major step: the action on $\Fx$ and $\me$}
Now, we are in the position to verify our main result. 

\begin{proof}[Proof of Main Theorem]
Recall that we assumed \eqref{FIX} and that our aim is to show that $\varphi$ is the identity map.
Observe that it is enough to prove that $\varphi$ acts identically on $\Fx$, as $\Fx$ is a weakly dense subset of $\me$ and $\varphi$ is continuous.
In order to do this we use induction on the cardinality of the support of $\mu\in\Fx$.
By Corollary \ref{corollary ketpont} our map $\varphi$ fixes all measures with an at most two-element support.
Let $k\in\mathbb{N}, k\geq 2$ and assume that we had already proved the following:
\begin{equation}\label{maxksupp}
\varphi(\nu) = \nu \qquad (\forall\; \nu\in\me, \# S_\nu \leq k).
\end{equation} 
Let us consider a measure
$$
\mu = \sum_{i=1}^{k+1}\lambda_i\delta_{x_i} \in \Fx
$$
where the $x_i$'s are pairwise different, $\sum_{i=1}^{k+1}\lambda_i=1$ and each $\lambda_i$ is positive.
Assume also that for every $1\leq i \leq k$ the point $x_i$ lies outside of the convex hull of $\{x_j\}_{j=i+1}^{k+1}$. 
Let us use the following notations in the sequel:
$$
\vartheta := \varphi(\mu) \quad\text{and}\quad \mu^{(i)} := \frac{1}{\sum_{j=i+1}^{k+1} \lambda_j}\cdot \sum_{j=i+1}^{k+1}\lambda_j\delta_{x_j} \qquad (0\leq i\leq k).
$$
By Proposition \ref{burok} we observe that the support of $\vartheta$ is contained in the convex hull of $S_\mu = \{x_j\}_{j=1}^{k+1}$.

Now, we prove step by step that each $x_i$ is an atom of $\varphi(\mu)$ with the same weight $\lambda_i$.
By \eqref{invariancia} we have $W_{\mu}\equiv W_{\vartheta}$, thus an application of Proposition \ref{burok} gives
\begin{align*}
\vartheta = \lambda_1\cdot \delta_{x_1} + (1-\lambda_1)\cdot \vartheta^{(1)},
\end{align*}
with a measure $\vartheta^{(1)}\in\me$ such that $x_1\notin S_{\vartheta^{(1)}}$ and $S_{\vartheta^{(1)}}$ lies in the convex hull of $S_{\mu} = \{x_j\}_{j=1}^{k+1}$.
Utilising Lemma \ref{tanu} and \eqref{maxksupp} for measures with supports of at most 2 elements we obtain
$$
W_{1-\lambda_1,\mu^{(1)}} \equiv W_{1-\lambda_1,\vartheta^{(1)}}.
$$
At this point, if $k$ was 2, then $\# S_{\mu^{(1)}} = 2$, thus by Lemma \ref{pi_s gyujtolemma} the measures $\mu^{(1)}$ and $\vartheta^{(1)}$ coincide, and therefore $\mu = \varphi(\mu)$ is yielded.
Otherwise, applying Lemma \ref{pi_s gyujtolemma} for the measures $\mu^{(1)}$ and $\vartheta^{(1)}$ gives
\begin{align*}
\vartheta = \lambda_1\cdot \delta_{x_1} + \lambda_2\cdot \delta_{x_2} + (1-\lambda_1-\lambda_2)\cdot \vartheta^{(2)},
\end{align*}
with a measure $\vartheta^{(2)}\in\me$ such that $x_2\notin S_{\vartheta^{(2)}}$ and $S_{\vartheta^{(2)}}$ lies in the convex hull of $S_{\mu^{(1)}} = \{x_j\}_{j=2}^{k+1}$.
Using Lemma \ref{tanu} and \eqref{maxksupp} for the case when the cardinality of the support is at most 3, we obtain
$$
W_{1-\lambda_1-\lambda_2,\mu^{(2)}} \equiv W_{1-\lambda_1-\lambda_2,\vartheta^{(2)}}.
$$
Iterating this procedure, the conclusion of the $(k-2)^{\text{nd}}$ step is the following: 
\begin{equation}\label{k-2}
W_{1-\sum_{i=1}^{k-2}\lambda_i,\mu^{(k-2)}} \equiv W_{1-\sum_{i=1}^{k-2}\lambda_i,\vartheta^{(k-2)}}
\end{equation}
where 
$$
\vartheta = \sum_{i=1}^{k-2}\lambda_i\delta_{x_i} + \bigg(1-\sum_{i=1}^{k-2}\lambda_i\bigg)\cdot \vartheta^{(k-2)}
$$
such that $\vartheta^{(k-2)}\in\me$, $x_{k-2}\notin S_{\vartheta^{(k-2)}}$ and $S_{\vartheta^{(k-2)}}$ lies in the convex hull of $S_{\mu^{(k-3)}} = \{x_{k-2},x_{k-1},x_k,x_{k+1}\}$.
Utilising Lemma \ref{pi_s gyujtolemma} for the measures $\mu^{(k-2)}$ and $\vartheta^{(k-2)}$ we get that
$$
\vartheta = \sum_{i=1}^{k-1}\lambda_i\delta_{x_i} + \bigg(1-\sum_{i=1}^{k-1}\lambda_i\bigg)\cdot \vartheta^{(k-1)}
$$
with some $\vartheta^{(k-1)}\in\me$, $x_{k-1}\notin S_{\vartheta^{(k-1)}}$ and $S_{\vartheta^{(k-1)}}$ lies in the convex hull of $S_{\mu^{(k-2)}} = \{x_{k-1},x_k,x_{k+1}\}$.
Furthermore, by Lemma \ref{tanu} and \eqref{maxksupp} we obtain
\begin{equation}\label{k-1}
W_{1-\sum_{i=1}^{k-1}\lambda_i,\mu^{(k-1)}} \equiv W_{1-\sum_{i=1}^{k-1}\lambda_i,\vartheta^{(k-1)}}.
\end{equation}
But since $\#S_{\mu^{(k-1)}} = 2$, Lemma \ref{pi_s gyujtolemma} and \eqref{k-1} imply $\mu^{(k-1)} = \vartheta^{(k-1)}$, and therefore we conclude $\varphi(\mu) = \mu$, completing the proof.
\end{proof}


\section{Concluding remarks}
We noted at the end of Section \ref{2} that it is possible to give a characterisation of surjective $\pi$-isometries on certain subsets of $\me$.
Namely, let $\mathcal{S}\subset\me$ be a weakly dense subset (possibly disjoint from $\Delta_X$), and assume that $\phi\colon\mathcal{S}\to\mathcal{S}$ is onto and satisfies
$$
\pi(\phi(\mu),\phi(\nu)) = \pi(\mu,\nu) \qquad (\forall\;\mu,\nu\in\mathcal{S}).
$$
Since $(\me,\pi)$ is a complete metric space, there exists a unique isometric extension $\varphi\colon\me\to\me$, i.e. $\varphi|_{\me} = \phi$. 
Clearly, $\varphi$ is a $\pi$-isometry which maps $\me$ into $\me$.
Observe that $\varphi[\me]$ is closed in $\me$.
On the other hand, as $\mathcal{S} = \varphi[\mathcal{S}] \subset \varphi[\me]$, we infer $\varphi[\me] = \me$.
Therefore $\varphi\colon\me\to\me$ is induced by a surjective isometry $\psi\colon X\to X$, whence we conclude the same for $\phi\colon\mathcal{S}\to\mathcal{S}$, i.e.
$$
\left(\phi(\mu)\right)(A) = \mu(\psi^{-1}[A]) \qquad (\forall\; \mu\in\mathcal{S}, A\in\bx).
$$

We proceed to mention some typical examples of weakly dense subsets of $\me$ (for which the above statement holds).
1) The set of all \emph{discrete Borel probability measures}, which is the collection of those $\mu\in\me$ that are concentrated on a countable subset of $X$.
2) The class of all \emph{continuous Borel probability measures}, i.e. those $\mu\in\me$ such that $\mu(\{x\}) = 0$ for every $x\in X$.
3) Let $n\in\mathbb{N}$, $X = \mathbb{R}^n$ and $\|\cdot\|$ be an arbitrary norm on $\mathbb{R}^n$. Since any two norms on $\mathbb{R}^n$ are equivalent, the Borel $\sigma$-algebra $\mathscr{B}_{\mathbb{R}^n}$ does not depend on $\|\cdot\|$. We say that $\mu\in\mathcal{P}_{\mathbb{R}^n}$ is an \emph{absolutely continuous Borel probability measure} if it is absolutely continuous with respect to the usual Lebesgue measure on $\mathbb{R}^n$. This set is clearly weakly dense in $\mathcal{P}_{\mathbb{R}^n}$, as every element of $\Fx$ can be approximated.

Next, as we have mentioned in the introduction, the most important special cases of our result are the following: 
1) when $X$ is an infinite dimensional, separable real Hilbert space;
2) when $X$ is the real Banach space $C([0,1])$; and
3) when $X$ is an $n$-dimensional Euclidean space ($n\in\mathbb{N}$).
We make some comments on how our proof could be modified in these cases.
In the first two cases the underlying Banach spaces are of infinite dimension, hence the support of any $\mu\in\Fx$ lies in a finite dimensional affine subspace.
In case of 1) the equivalence in Proposition \ref{burok} can be done for every element $\xk$ in $S_\mu$ by choosing a half-line $\mathfrak{e}$ orthogonal to that affine subspace.
Therefore the proof becomes much simpler as we immediately obtain that every $\mu\in\Fx$ is fixed by $\varphi$.
A similar argument simplifies the proof for general strictly convex infinite dimensional separable Banach spaces.
In case of 2) the space is of infinite dimension but the norm is not strictly convex.
Despite of this obstacle the proof still can be shortened by utilising the Lindenstrauss--Troyansky theorem \cite{Bourgain,Lindenstrauss,Troyanski}.
Namely, if $\mu\in\Fx$ and $S_\mu$ is contained in the kernel of a strongly exposing functional (for the definition see e.g. \cite{Bourgain}), then the equivalence part of Proposition \ref{burok} can be verified for every element $\xk$ in $S_\mu$. 
Since by the Lindenstrauss--Troyansky theorem it is easy to see that every $\mu\in\Fx$ can be weakly approximated by such measures, we easily complete the proof of the Main Theorem in this case too.
It seems that for finite dimensional spaces, even for the case of 3), we really have to do the whole procedure presented in Section \ref{3}, or at least we are not aware of any shortening possibilities.

Finally, we note that throughout Section \ref{3} there were some parts where we considered general complete and separable metric spaces.
But later on most of our techniques required that the underlying space had a linear structure.
In our opinion it would be interesting to find a characterisation of all surjective $\pi$-isometries in the setting of other special (but still general enough) kinds of complete separable metric spaces.


\end{document}